\documentclass[12pt, reqno, a4paper]{amsart}
\usepackage{amsmath,mathrsfs}
\usepackage{amssymb}
\usepackage{color}
\usepackage{calligra}
\usepackage{dsfont}
\usepackage{pstricks,pst-node}
\usepackage[latin1]{inputenc}
\usepackage{bbm}
\usepackage{enumitem}

\newcommand\NoBlackBoxes{\global\overfullrule0pt}
\NoBlackBoxes

\newcommand{\N}{\mathbb{N}}

\makeatletter
\let\serieslogo@\relax
\let\@setcopyright\relax

\newtheorem{definition}{Definition}[section]
\newtheorem{theorem}[definition]{Theorem}
\newtheorem{proposition}[definition]{Proposition}
\newtheorem{lemma}[definition]{Lemma}
\newtheorem{rem}[definition]{Remark}
\newtheorem{rems}[definition]{Remarks}

\newenvironment{remark}[1][Remark]{\begin{trivlist}
\item[\hskip \labelsep {\bfseries #1}]}{\end{trivlist}}

\renewcommand{\P}{{\mathbb{P}}}
\newcommand{\E}{{\mathbb{E}}}
\newcommand{\V}{{\mathbb{V}}}
\newcommand{\R}{{\mathbb{R}}}

\newcommand{\Fc}{{\mathcal{F}}}

\newcommand{\Gc}{{\mathcal{G}}}
\newcommand{\Oc}{{\mathcal{O}}}
\newcommand{\Cc}{{\mathcal{C}}}

\renewcommand{\l}{\lambda}

\newcommand{\Var}{\mathbb{V}}

\renewcommand{\a}{\alpha}

\renewcommand{\epsilon}{\varepsilon}
\renewcommand{\phi}{\varphi}
\newcommand{\Cmax}{C_{\max}}
\newcommand{\Ckn}{C_{\leq k_N}}
\newcommand{\Ckna}{|C_{\leq k_N}|}

\newcommand{\Cs}{C_{second}}
\renewcommand{\b}{\beta}

\numberwithin{equation}{section}

\begin{document}

\setcounter{page}{1}

\title[Moderate deviations for random hypergraphs]{Moderate deviations for the size of the giant component in a random hypergraph}
%
%

\author[Jingjia Liu]{Jingjia Liu}
\address[Jingjia Liu]{Fachbereich Mathematik und Informatik,
Universit\"at M\"unster,
Einsteinstra\ss e 62,
48149 M\"unster,
Germany}

\email[Jingjia Liu]{jingjia.liu@uni-muenster.de}

\author[Matthias L\"owe]{Matthias L\"owe}
\address[Matthias L\"owe]{Fachbereich Mathematik und Informatik,
Universit\"at M\"unster,
Einsteinstra\ss e 62,
48149 M\"unster,
Germany}

\email[Matthias L\"owe]{maloewe@uni-muenster.de}


\date{\today}



\newcommand{\wlim}{\mathop{\hbox{\rm w-lim}}}
\newcommand{\na}{{\mathbb N}}
\newcommand{\re}{{\mathbb R}}

\newcommand{\vep}{\varepsilon}

\begin{abstract}
We prove a moderate deviations principles for the size of the largest connected component in a random $d$-uniform hypergraph. The key tool is a version of the exploration process, that is also used to investigate the giant component of an Erd\"os-R\'enyi graph, a moderate deviations principle for the martingale associated with this exploration process, and exponential estimates.
\end{abstract}

\maketitle

\section{Introduction}
The research on random graphs was initiated by Erd\"os and R\'enyi, see \cite{ER59}, \cite{ER61}. Though it was originally motivated by questions from graph theory, random graphs quickly developed into an independent field with applications in many areas such as physics, neural networks, telecommunications, or the social sciences. Despite the fact that some of these applications ask for random graphs with a given degree distribution (see e.g. \cite{durrett_graphs}, \cite{vdH} for very readable surveys, or \cite{LV14}, for a recent application), the by far most popular model of a random graph still is the Erd\"os-R\'enyi graph. In this graph, one realizes all possible connections between $N$ vertices $V=\{1, \ldots, N\}$  independently with equal probability $p$. This models is referred to as $\Gc(N,p)$.

The corresponding random hypergraph model is the model $\Gc^d(N,p)$. Here $d \ge 2$ is an integer number that denotes the cardinality of the hyperedges (the case $d=2$ is nothing but the ordinary Erd\"os-R\'enyi graph). Thus a realization of $\Gc^d(N,p)$  will be a hypergraph $G=(V,E)$, where all the edges in $E$ are subsets of $V$ with cardinality $d$. Moreover, in $\Gc^d(N,p)$ all hyperedges of cardinality $d$ are selected independently with probability $p$. One of the most striking first results about $\Gc(N,p)$ is, that there is a sharp phase transition in the size of the largest connected component: If $p= \l/N$ and $\l <1$, then the largest component will be of size $\mathcal{O}(\log N)$, while for $\l >1$, the component is of order $\mathcal{O}(N)$, both with probability converging to 1. In the latter case, the size of the largest component with high probability is of order $\rho_\l N + o(N)$, where $\rho_\l$ satisfies
\begin{equation} \label{equa_rho2}
1-\rho_\l = e^{-\l\rho_\l}
\end{equation}
(here we say that
the random (hyper)graph $\Gc^d(N, p)$ enjoys a certain property $A$ with high probability (w.h.p.),
if the probability that A holds in $\Gc^d(N, p)$ converges to 1 as $N$ tends to infinity).
A very detailed study of this and many other phenomena concerning this phase transition can be found in \cite{giant_component} or \cite{vdH}. 
The corresponding result for the d-regular random hypergraph models $\Gc^d(N,p)$ were shown in \cite{S-PS85}, \cite{KL02}, and \cite{C-OMS07}: If for some $\vep>0$ we have $(d-1) \binom{N-1}{d-1} p < 1-\vep$ the resulting hypergraph consist of components of order $\mathcal{O}(\log N)$, while for  $(d-1) \binom{N-1}{d-1} p > 1+\vep$ there is a unique giant component of  size $\mathcal{O}(N)$. To make this more precise, we need a number of definitions. We set
\begin{equation}\label{pdef}
p=\frac{\l (d-2)!}{N^{d-1}}.
\end{equation}

For each fixed $\l>1$, we define the dual branching process parameter $\l_*<1$  by the equation
\begin{equation*}
\l_* e^{-\l_*}= \l e^\l.
\end{equation*}
In case $d=2$, we specify $\rho_\l$ given by \eqref{equa_rho2} as $\rho_\l=:\rho_{2,\l}$, whereas for $d\geq 3$, we define $\rho_{d,\l}$ by the equation
\begin{equation}\label{equa_rho2rhoD}
1-\rho_{d,\l}=\left(1-\rho_\l\right)^{1/(d-1)}.
\end{equation}
It can be checked that $\rho_{d,\l}$ satisfies
\begin{equation}\label{def_lambda*}
\l_*= \l \left( 1- \rho_{d,\l} \right)^{d-1}.
\end{equation}

For fixed $d$ we abbreviate $\overline{\rho_\l}=\rho_{d,\l}$.
The role of $\overline{\rho_\l}$ is that it determines the asymptotic size of the giant component. Indeed, 
if $\l > 1$, it has been shown in \cite{ER59} and \cite{C-OMS07} that the unique giant component is of size $\overline{\rho_\l}N+o(N)$ with high probability.
This statement can be regarded as a law of large numbers for the size of the giant component, which we will call henceforth $C_{\max}$. 

Moreover it was shown that and $\overline{\rho_\l}$ also can be written as the unique solution of the transcendental equation (cf.\ \cite{BC-OK10}):
\begin{equation}\label{rho_alt}
1-\overline{\rho_\l} = \exp\left(-\frac{\lambda}{d-1}\left(1-(1-\overline{\rho_{\l}})^{d-1}\right)\right).
\end{equation}
Combining \eqref{equa_rho2rhoD} with \eqref{def_lambda*}, one sees that \eqref{equa_rho2} is indeed the case $d=2$ of \eqref{rho_alt}.
Note that both $\rho_\l$ and $\overline{\rho_\l}$ depend on $d$, which is suppressed in the notation.

Let us assume for the rest of the paper that we are in the supercritical regime, i.e.\ 
$$
(d-1) \binom{N-1}{d-1} p > 1+\vep \quad\mbox{ for some $\vep>0$},
$$
where the precise conditions on $\epsilon$ will be given later explicitly. Note that this is equivalent to assuming that $\lambda >1+\epsilon$.


For both, random graphs and random hypergraphs, fluctuations around the aforementioned law of large numbers for the size of $\Cmax$ were investigated. A Central Limit Theorem (CLT, for short) for the size of $\Cmax$ in $\Gc(N,p)$ was proved e.g. in \cite{BBF00}, for a nice proof we also refer to \cite{vdH}, Section 4.5. Large deviations in the same situations go back to O'Connell in his nice paper \cite{NOC98}, while moderate deviations were investigated in \cite{AL11}. The corresponding CLT for $\Cmax$ is $\Gc^d(N,p)$ was established in \cite{BC-OK10} using Stein's method. In \cite{BC-OK14} a local CLT is proved, even for the joint distribution of the number of vertices and edges in $\Cmax$. Another way to prove a CLT that uses the so-called exploration process and is reminiscent to the proof for random graphs given in \cite{vdH}  was introduced by Grimmett and Riordan \cite{bollobas_riordan2}.

The aim of the present paper is  to establish moderate deviations results for the number of vertices in $\Cmax$ for the case of the random hypergraph model $\Gc^d(N,p)$. To this end, we will modify the exploration process for hypergraphs introduced in \cite{bollobas_riordan2} in such a way that is resembles the exploration process used in \cite{AL11}.

In order to formulate our main theorems we need to recall that a sequence of real valued random variables
$(Y_n)$
obeys a large deviation principle (LDP) with speed $a_n$ and good
rate function $I(\cdot):\R\to \R^+_0\cup \{+\infty\}$ if
\begin{itemize}
\item For every $L \in  \R^+_0$, the level sets of $I$ denoted by  
$N_L:=\{x\in \R: I(x) \le L\}$, are compact
\item
For every open set $G \subseteq \R$ it holds
\begin{equation}
\liminf_{n \to \infty} \frac 1 {a_n} \log \P(Y_n \in G)\ge -\inf_{x\in G} I(x).
\end{equation}
\item
For every closed set $ A\subseteq \R$ it holds
\begin{equation}
\limsup_{n \to \infty} \frac 1 {a_n} \log \P(Y_n \in A)\le -\inf_{x\in A} I(x).
\end{equation}
\end{itemize}
As announced, in this paper we will prove a moderate deviation principle (MDP) for $|\Cmax|$ (which is a function of $N$). Formally, there is no distinction between an MDP and an LDP.
Usually, an LDP lives on the scale of a law of large number type
ergodic phenomenon, while MDPs describe the probabilities on a scale
between a law of large numbers and some
sort of CLT. For both, large deviation principles and MDPs
the three points mentioned above serve as a definition.

Having this in mind our central result reads as: 
\begin{theorem}[MDP for the size of the giant component in $\Gc^d(N,p)$]\label{main_theorem}
Let $\frac 12 <\alpha < 1$. For each $d \ge 3$ , set $p = \l \frac{(d-2)!}{N^{d-1}}$ with $\l=1+\epsilon$. Assume that $\epsilon=\mathcal{O}(1)$, 
as well there exists $\iota>0$ such that $ \epsilon^3 N^\tau\to\infty$ where 
\begin{equation}\label{cond:tau}
\tau=\min\lbrace \frac{1}{2},\, 2-2\a-\iota \rbrace.
\end{equation}

Then the sequence of random variables $(|\Cmax|-\overline{\rho_\l}N)/N^\a$ satisfies an MDP in $\Gc^d(N,p)$ with speed $N^{2\a-1}$ and rate function 
\begin{equation}\label{rategraph}
J(x)=\frac{x^2\left( 1-\l_*\right)^2}{2c}.
\end{equation}
Here $c=  \l (1-\overline{\rho_\l})^2-\l_* (1-\overline{\rho_\l})+\overline{\rho_\l}(1-\overline{\rho_\l})$ and 
$\l_*= \l \left( 1- \rho_{d,\l} \right)^{d-1}$ as in \eqref{def_lambda*}.
\end{theorem}
\begin{rems}
\normalfont{
\begin{enumerate}
\item For $d=2$ this result is contained in \cite{AL11}.
\item The asymptotic notation should be understood as $N\to\infty$. We use $X=\mathcal{O}(Y)$, if there is an $ M>0$ such that $\limsup\limits_{N\to\infty}|\frac{X}{Y}|\leq M$, and $X=o(Z)$, if there exists $c(N)$ such that $X\leq c(N)Z$, where $c(N)\to 0$ as $N\to\infty$. Furthermore, if such quantities $M, c(N)$ depend on some parameters, we will indicate this by subscripts, e.g.\ $X=\mathcal{O}_\rho(Y)$ meaning $M=M(\rho)$.
\item From the proof of Lemma \ref{lem_errorESTrn} one might get the impression that requiring the (slightly more natural) condition
$ \epsilon^3 N^\tau\to\infty$ with $\tau=\min\lbrace \frac{1}{2},\, 2-2\a \rbrace$ would be enough. However, in the proof of Lemma \ref{lemma5.4} we need the slightly stronger condition \eqref{cond:tau}.
\end{enumerate}}
\end{rems}

The rest of this paper is organized as follows: In Section 2 we give a short introduction to the exploration process which will be used in Sections 4, 5 and 6 to prove Theorem \ref{main_theorem}.  Briefly this exploration process starts with a number $k=k_N$ of vertices and investigates the union of its connected components. If $k_N$ is chosen appropriately, this union coincides with the giant component of the hypergraph up to negligible terms. On the other hand, the size of this union can be controlled by a martingale underlying the exploration process. In Section 3 we prove an MDP for this martingale. In Sections 4,5, and 6 we will see that indeed this MDP helps to show our main Theorem \ref{main_theorem}.

\section{An exploration process on hypergraphs}
The aim of this section is to introduce an exploration process to investigate the components of a hypergraph. This exploration process is inspired by the corresponding process for graphs as defined e.g.\ in \cite{vdH}. A similar, yet slightly different process for hypergraphs was introduced in \cite{bollobas_riordan2}. We will also use results from this paper. 

We start by taking the given enumeration of the vertices from $1, \ldots, N$. Vertices during this exploration process will get one of three labels: active, unseen, or explored. At time $t$ the sets of active, unseen, or explored vertices will be denoted by $A_t$, $U_t$, and $E_t$, respectively. We will start by declaring the first $k=k_N$ vertices active and the rest unseen. Now, in each step of the process, the first active vertex (with respect to the given enumeration) is selected and declared \textit{explored}. At the same time, all of its unseen neighbors are set active. The process terminates when there are no active vertices.
If we denote by $C_{\leq k}$ the union of the connected components of the first $k$ vertices, then, at the end of the process
all vertices in $C_{\leq k}$  are explored and all the others are unseen.
We remark the following

\begin{rems}\label{r:exp:CkSt}
\normalfont{
\begin{enumerate}
\item For two sequences $X_N$ and $Y_N$, we write $X_N\sim Y_N$, if the limit $\lim_{N\to \infty} X_N/Y_N$ exists and equals to 1.
\item Obviously, since we add an explored vertex in every step
\[ |C_{\leq k_N}| = \min\{ t\in\N : A_t=0\}.\]

\item By construction, $A_0 = k_N$  (to be specified later) and, for all $t$ with $A_{t-1} > 0$, one has
$$
A_t = A_{t-1} + \eta_t -1.
$$
Here $\eta_t$ is the number of unseen vertices that are set active in the $t-$th step and
$A_t = 0$ if $A_{t-1}=0$.
\item Consider the distribution of $A_t$ when we are investigating $\Gc^d(N,p)$ with $p = \l \frac{(d-2)!}{N^{d-1}}$ and $\l>1$. For each $t\in\N$ after the $t-$th step, with $s$ active and $N-t-s$ unseen vertices $u$, for each unseen vertex $u$ there are exactly
$$
\nu_{t+1} = \binom{N-t-2}{d-2}
$$
potential edges that contain $u$ and the vertex we are about to explore, but none of the vertices we have already explored. Hence the probability that $u$ becomes active during step $t+1$ is given by
$$ 
\pi= \pi_t= 1-(1-p)^{\nu_{t+1}}.
$$ 
Note that for times $t \ll N$ and with our scaling of $p$ one has
$$
\pi \sim 1-e^{-\frac \l N}
$$ 
which is the same scaling as in the case $d=2$. On the other hand, $$
\pi_t = p \nu_{t+1}+ \mathcal{O}(p^2 \nu_{t+1}^2)= \lambda (N-t)^{d-2} N^{-d+1}+ \mathcal{O}(\frac{1}{N^2}).
$$
This implies that $\eta_t$ given $A_{t-1} = s$ is distributed like $\sum_{i=1}^{N-(t-1)-s} Y_i^t$, where each of the $Y_i^t$ is an indicator with success probability $\pi_{t}$.  Note that the $Y_i$ are {\it not} independent, which establishes the major difference between the case $d=2$ (when the $Y_i$ are obviously independent) and the cases $d\ge 3$.
\end{enumerate}}
\end{rems}

To simplify matters, we will change our process slightly and call instead
\[ A_t = A_{t-1} + \eta_t -1\]
for all $t\in\N$ the exploration process. Of course, this process agrees with the one previously considered up to the first time the process hits 0.
We will follow the ideas in \cite{bollobas_riordan2} and rewrite $A_t$ (up to small errors) as the sum of a deterministic process and a martingale. This also motivates our study of moderate deviations of martingales with $n$ dependent martingale increments in the next section.

To this end, let $$D_t:= \E[\eta_t-1| \mathcal{F}_{t-1}]$$ where $\mathcal{F}_{t}$ is the $\sigma$-Algebra generated by the random variables $A_0, \eta_1, \ldots, \eta_t$. From the above we learn that
$$
 \E[\eta_{t+1}| \mathcal{F}_{t}]= U_t \pi_1 = U_t p \nu_{t+1} + \mathcal{O}(\frac 1N).
$$
For later use we also recall that in \cite{bollobas_riordan2} it was shown that with
$$
\pi_2:= 1-(1-p)^{\binom{N-t-3}{d-3}}\sim p^{\binom{N-t-3}{d-3}}\sim \lambda (d-2)(N-t)^{d-3}N^{-d+1}
$$
and $\pi_3 \sim \pi_1^2$ we have that
\begin{eqnarray}\label{var}
\V[\eta_{t+1}| \mathcal{F}_{t}]&=& U_t(U_t-1)(\pi_2+\pi_3)+U_t\pi_1-(U_t \pi_1)^2 \nonumber\\
&\sim& U_t^2 \pi_2 +U_t\pi_1\nonumber\\
&\sim& \lambda (d-2) (1-\frac t N)^{d-3} \frac {U_t^2}{N^2}+ \l (1-\frac t N)^{d-2 } \frac {U_t}{N}
\end{eqnarray}

From these computations we obtain
\begin{eqnarray*}
D_{t+1}&=& U_t p \nu_{t+1}-1 + \mathcal{O}(\frac 1N)\\
&=& \alpha_{t+1} (N-t-A_t)-1+ \mathcal{O}(\frac 1N)
\end{eqnarray*}
where 
$$
\alpha_t= p \nu_t= p \binom{N-t-1}{d-2} .
$$  
Now set 
\begin{equation}\label{def_Delta}
\Delta_{t+1}:= A_{t+1}- A_t -D_{t+1} =\eta_{t+1}-\E[\eta_{t+1}| \mathcal{F}_t]
\end{equation}
such that
\begin{equation}\label{def_At+1}
\begin{split}
A_{t+1} 
=& A_t + D_{t+1}+ \Delta_{t+1}\\
=& (1-\alpha_{t+1})A_t+\alpha_{t+1}(N-t) -1 +\Delta_{t+1}+\mathcal{O}(\frac 1N).
\end{split}
\end{equation}
By definition we have $\E[\Delta_{t+1}| \mathcal{F}_{t}]=0$, thus $(\Delta_t)_t$ is a martingale difference sequence. In particular, a simple bound for the variance of $(\Delta_t)_t$ is given by
\begin{lemma}\cite[Lemma 8]{bollobas_riordan3}
\label{lem_conVara.s}
Let $p$ and $\l>1$ be given as in Theorem \ref{main_theorem}. Then there is a constant $M>0$ such that for all $1\leq t \leq N$, we have
$$
\Var(\Delta_t \vert \Fc_{t-1})\leq M
$$
with probability 1.
\end{lemma}

Obviously, the process $A_t$ is a key quantity for the analysis of the size of $C_{\le k_N}$ (and hence for the size of $\Cmax$)
We want to approximate it by the sum of a deterministic sequence and a martingale. To this end, we define 
\begin{align*}
x_0 &=0 \\
x_{t+1}&= (1-\alpha_{t+1})x_t+\alpha_{t+1}(N-t) -1.
\end{align*}

Then, with $A_0=k_N$, define
$$
A_{t+1}-x_{t+1}= (1-\alpha_{t+1})(A_t-x_t) + \Delta_{t+1}+\vep_{t+1}
$$
where $\vep_{t+1}$ is a shorthand for the error term at level $t+1$, which is of order $\mathcal{O}(\frac 1N)$ (cf. \eqref{def_At+1} and for more details we refer the reader to \cite[(10)]{bollobas_riordan2}).
So, if we set
$$
\beta_t := \prod_{i=1}^t (1-\a_i)
$$
we arrive at
$$
A_t-x_t = \sum_{i=1}^t \frac{\beta_t}{\beta_i}(\Delta_i+\vep_i).
$$
By defining
\begin{equation}\label{mart}
S_t:=  \sum_{i=1}^t \frac{1}{\beta_i}\Delta_i,
\end{equation}
we observe that $(S_t)$ is a martingale. Thus our desired approximation is given by
\begin{equation}\label{Approx.Proc}
 \tilde A_t:= x_t+ \beta_t S_t,   
\end{equation}
and we have by \cite[Lemma 3]{bollobas_riordan2} that
\begin{equation}
\begin{split}
\label{eq_Approx.errorbound}
|A_t-\tilde A_t| 
=& |A_t-x_t- \beta_t S_t| \\
=& |\sum_{i=1}^t \frac{\beta_t}{\beta_i}(\Delta_i+\vep_i)-\beta_t \sum_{i=1}^t \frac{1}{\beta_i}\Delta_i |\\
=& |\sum_{i=1}^t \frac{\beta_t}{\beta_i}\vep_i|
=\mathcal{O}\big(\frac { t}{N} \big)
\end{split}
\end{equation}
uniformly in $1\leq t\leq N$.

Note that the behaviour of $x_t$ can be determined as well (see \cite[(15)]{bollobas_riordan2}). Indeed, define $g_{d, \l}$ as
\begin{align}\label{eq_def.g}
g_{d, \l}(x)= 1-x-\exp\left(-\frac \l{d-1}\left(1-(1-x)^{d-1}\right)\right)
\end{align}
and
\begin{equation}\label{eq_def.f}
f(t):= f_{N, d, \l}(t):= Ng_{d, \l}(\frac tN).
\end{equation}
Then, we obtain uniformly in $0\leq t \leq N$,
\begin{align}\label{eq_trajectoryxt}
x_t= f(t)+\Oc(1).
\end{align}
In a nutshell the idea is now the following: The first time $t$, that $A_t$ is $0$ is the size of the union of the connected components of the first $k_N$ vertices. If $k_N$ is chosen large enough, this union will contain the giant component with overwhelming probability (where we say that an event has ''overwhelming probability'' when the probability of the complement of the event is negligible on the chosen moderate deviation scale). On the other hand, also with overwhelming probability, there is only one component with a size larger than $N^\xi$, if $\xi<1$ is chosen appropriately. Hence the smaller components do not count on a moderate deviations scale, and on that scale the size of the union of the connected components of the first $k_N$ vertices is the size of the largest component with overwhelming probability. On the other hand, as observed above, we may safely replace $A_t$ by $\tilde A_t$ when considering these quantities on a moderate deviation scale. That is to say
\begin{equation}\label{eq_simMDP1}
\P(A_t/N^\a \ge x) \sim \P(\tilde A_t/N^\a \ge x)
\end{equation}
for any $x$ and $\a >0$ and likewise 
\begin{equation}\label{eq_simMDP2}
\P(A_t/N^\a \le x) \sim \P(\tilde A_t/N^\a \le x).
\end{equation}
Moreover, notice that the stochastic behaviour of $\tilde A_t$ is governed by the martingale $(S_t)$ defined in \eqref{mart}. We will therefore analyze its moderate deviations in the next section.

\section{Moderate deviations for the martingale $S_t$ defined in \eqref{mart}}
As we will see in Section 4 the moderate deviations for the size of the giant component can be played back to the moderate deviations for the for the martingale $(S_t)$ defined in \eqref{mart}, in this section we will prove a moderate deviations principle for this martingale. Our main tool is the G\"artner-Ellis theorem \cite[Theorem 2.3.6]{dembozeitouni}. Note that we cannot simply quote MDPs for martingales from \cite{dembo_mart_mdp} or \cite{EL17}, because in our context the distributions of the martingale differences does depend on $N$.
The central result of this section is
\begin{theorem}\label{MDPmart}
Consider the process $(S_n)$ defined in \eqref{mart}. For  $\zeta \in (-\infty, \infty)$, and $\frac 12 < \a < 1$, put $\gamma(N)= \overline{\rho_\l} N + \zeta N^\a$ and assume for simplicity that $\gamma(N)$ is an integer to avoid rounding.
Then, for any choice of $\zeta$ the sequence of random variables 
$$
\frac{\beta_{\gamma(N)} S_{\gamma(N)}}{N^\alpha} 
$$ 
satisfies an MDP with speed $N^{2\alpha-1}$ and rate function $I(x)=x^2/2c$ where $c$ is given by
\begin{equation}\label{eq_c}
c= c_{d, \l}=  \l (1-\overline{\rho_\l})^2-\l_* (1-\overline{\rho_\l})+\overline{\rho_\l}(1-\overline{\rho_\l}) .
\end{equation}
Here we write $\overline{\rho_\l}$ for $\rho_{d,\l}$ and $\l_*$ is the dual branching parameter given in \eqref{def_lambda*}.
\end{theorem}

The key idea will be to employ the G\"artner-Ellis theorem  \cite[Theorem 2.3.6]{dembozeitouni}. To this end we need to study the moment generating function of $S_n$ on the level of moderate deviations, i.e. we need to establish the existence of
$$
\lim_{N \to \infty} \frac 1 {N^{2\a-1}} \log \E \left(\exp\left(t \frac{N^\a}{N} \beta_{\gamma(N)} S_{\gamma(N)}\right)\right)
$$				
for $t \in \R$. We will expand the moment generating function into a Taylor series up to the third order. However, to compute this we need some preparation. Recall our definitions \eqref{def_Delta} and \eqref{mart}, from which we obtain $\Delta_{t+1} =\eta_{t+1}-\E[\eta_{t+1}| \mathcal{F}_t]$ and $S_t=  \sum_{i=1}^t  \Delta_i/\beta_i $.

\subsection{Moments of $\eta_i$}
The essential point in the Taylor expansion is that the conditional variances of $\Delta_i$ and thus of $\eta_i$ depend on the number of unseen vertices at time $i$, $U_i$. We will therefore show a rough concentration result for $U_i$. To this end, recall that $\eta_i$ is the number of unseen vertices that are set active in the $i$-th step. 
Thus, it holds
$$
\eta_i = \sum_{ j\in U_i} \mathbbm{1}_{\left\lbrace \substack{\mbox{ $\exists$ an unexplored hyperedge containing $j$ and the currently active vertex}} \right\rbrace}
$$ 
and for any $k \ge 2$
$$
\eta_i^k = \sum_{\substack{j_{i_1}, \ldots, j_{i_k} \in U_i\\ \lbrace i_1,\ldots, i_k \rbrace\subset \lbrace 1,\ldots,N\rbrace}} \mathbbm{1}_{\left\lbrace j_{i_1}, \ldots, j_{i_k} \mbox{ are activated in step } i \right\rbrace}.
$$

Assume from stage $i$ to $i+1$, the vertex $v_i$ is being explored. There are various ways to activate $j_1, \ldots j_k$. 
Without loss of generality, we assume that the $j_i$ are pairwise different; otherwise this problem is reduced to estimating the lower moment of $\eta_i$.
One needs to activate a set of hyperedges $e_1, \ldots, e_m$ such that $j_1, \ldots, j_k$ are contained in these hyperedges. Because we have to choose the remaining vertices of the hyperedge from the unexplored vertices, the probability that fixed $l$ of the vertices are contained in one hyperedge is given by
$$
\pi_l = 1- (1-p)^{\binom{N-i-l-1}{d-l-1}} \sim p\binom{N-i-l-1}{d-l-1}\sim \l (d-2) \cdots (d-l) (N-i)^{d-l-1} N^{1-d}.
$$
That is, if $l_1, \ldots l_m$ sum up to $k$, then $\prod \pi_l \leq   D\l^m N^{-k}$ for some constant $D$. On the other hand, there is a constant $C$ such that there at most $C^{k/d+1}$ ways to write $k$ as a sum of integers at most $d-1$. Altogether with \eqref{var} this gives for $k=3$ and any $1\le i\le N$ 
\begin{equation}\label{bern_cond}
\E \left\lbrack  |\eta_i |^3 \vert\mathcal{F}_{i-1}\right\rbrack \le L \Var(\eta_i \vert \mathcal{F}_{i-1}) 
\end{equation}
for some constant $L>0$ as well as
\begin{equation}\label{bern_cond2}
\E \left\lbrack  |\Delta_i |^3 \vert\mathcal{F}_{i-1}\right\rbrack \le L \Var(\Delta_i \vert \mathcal{F}_{i-1}). 
\end{equation}

\subsection{Exponential estimates}
Moreover, we will need a Hoeffding-Azuma-type inequality for $S_N$ (e.g.\ \cite[Lemma 12]{bollobas_riordan3}):

For a constant $c_3>0$, it holds
\begin{equation*}
\P( \max_{1 \le t \le N} S_t \ge y ) \le \exp\left(-c_3 y^2/N\right).
\end{equation*}
In particular, taking $y=N^\beta$ for $\beta>0$ yields
\begin{equation}\label{max_doob}
\P( \max_{1 \le t \le N} S_t \ge N^\beta ) \le \exp\left(-c_3N^{2\beta-1}\right).
\end{equation}
(This could, in fact, also be proved using the results in \cite{delapena} together with our above considerations).  

\subsection{Taylor expansion of $\Delta_i$ up to third order}
As we know 
$$U_t=N-t-A_t, \qquad A_t= \tilde A_t + \mathcal{O}(\frac tN), \qquad \tilde A_t= x_t + \beta_t S_t
$$ 
and the trajectory of $x_t$ is given by \eqref{eq_def.f}. Denote (cf. \cite{bollobas_riordan2}, equation(22))
\begin{equation}\label{eq_ui}
u_i= N \exp\left(-\frac{\lambda}{d-1}\left(1-(1-\frac iN)^{d-1}\right)\right).
\end{equation}
Hence, we have 
$$
U_i-u_i
=-\beta_i S_i  + \Oc(1).
$$

Fix a $\beta>0$ with $\frac 12 < \a<\beta <1$ and take a deterministic sequence $b_N$ with $b_N\to\infty$ and $b_N=\Oc(N^\beta)$, thus $b_N=o(N)$. Let us define the event 
$$
\Sigma_N =\lbrace \forall 1\leq i\leq N\colon |U_{\gamma(i)}-u_{\gamma(i)}|\leq b_N  \rbrace,
$$
which has probability at least $1-\exp( - \eta N^{2\b-1} )$ for an appropriate $\eta>0$ by \eqref{max_doob}. Indeed, the event $\lbrace \exists i\colon 1\leq i\leq N,  \mbox{ such that } |U_{\gamma(i)}- u_{\gamma(i)}|> b_N \rbrace$ has probability at most $\exp(-c_4N^{2\beta-1})$ for some constant $c_4>0$. ($c_4$ differs from $c_3$ in \eqref{max_doob} by at most an absolute constant factor.) 

We use the martingale property to get
\begin{equation} \label{equ_mgfEST1}
\begin{split}
&\E \left(\exp\left(t \frac{N^\a}{N} \beta_{\gamma(N)} S_{\gamma(N)}\right)\right)\\
=& \E \left( \E\left[ \exp\left(t \frac{N^\a}{N}\beta_{\gamma(N)} S_{\gamma(N)} \right)\bigg \vert  \mathcal{F}_{\gamma(N)-1}\right] \right)\\
=&\E\left(\left(\exp\left(t \frac{N^\a}{N} \beta_{\gamma(N)} S_{\gamma(N)-1}\right)\right) \E\left[\exp\left(t \frac{N^\a}{N} \Delta_{\gamma(N)} \right) \bigg|\mathcal{F}_{\gamma(N)-1}\right]\right).
\end{split}
\end{equation}
By expanding up to third order, using that $\E[\Delta_{i+1} |\mathcal{F}_i]=0$, for all $i$, applying \eqref{var} as well as the crude bounds on the third moment derived in \eqref{bern_cond}, we obtain
\begin{equation} \label{equ_mgfEST2}
\begin{split}
& \E\left[\exp\left(t N^{\a-1} \frac{\beta_{\gamma(N)}} {\beta_{\gamma(i)+1}}\Delta_{\gamma(i)+1} \right) \bigg|\mathcal{F}_{\gamma(i)}\right] \\
=& 1+ \frac {t^2}{2} N^{2\a-2} \frac{\beta_{\gamma(N)}^2}{\beta_{\gamma(i)+1}^2} \left(\l(d-2) \left(1-\frac {\gamma(i)}{N}\right)^{d-3}\frac{U_{\gamma(i)}^2}{N^2}+ \l \left(1-\frac {\gamma(i)} N \right)^{d-2 } \frac {U_{\gamma(i)}}{N}\right)\\
& +\mathcal{O}\left(t^3 N^{3\a-3}\frac{U_{\gamma(i)}}{N} \right).
\end{split}
\end{equation}
Here we used the fact $U_n \le n \le N$ and $\beta_t = \prod_{i=1}^t (1-\a_i)$. We arrive at
\begin{align*}
&\E\left[  \exp\left( t N^{\a-1} \frac{\beta_{\gamma(N)}}{\beta_{\gamma(i)+1}} \Delta_{\gamma(i)+1} \right) \mathds{1}_{\Sigma_N} \bigg| \mathcal{F}_{\gamma(i)} \right] \\
=& \exp\bigg( \frac {t^2}{2} N^{2\a-2} \frac{\beta_{\gamma(N)}^2}{\beta_{\gamma(i)+1}^2} \left(\l(d-2) \left(1-\frac {\gamma(i)}{N}\right)^{d-3} \frac{u_{\gamma(i)}^2}{N^2}+ \l \left(1-\frac {\gamma(i)} N \right)^{d-2 } \frac {u_{\gamma(i)}}{N}\right)\\
&+ \mathcal{O}(N^{3\a-3}) \bigg)(1+o(1)).
\end{align*}
Redoing this conditioning $\gamma(N)$ times, we finally see that
\begin{equation}\label{eq_TaylorGammaN}
\begin{split}
&\E \left[\exp\left(t \frac{N^\a}{N}\beta_{\gamma(N)} S_{\gamma(N)}\right) \mathds{1}_{ \Sigma_N}\right]\\
=& \exp\bigg(\frac {t^2}2 N^{2\a-2} \sum_{i=0}^{\gamma(N)-1}\frac{\beta_{\gamma(N)}^2}{\beta_{i+1}^2}\left(\l(d-2)\left(1-\frac{i}{N}\right)^{d-3}\frac{u_i^2}{N^2}+ \l \left(1-\frac {i}{N}\right)^{d-2 }\frac {u_i}{N}\right)\\
&+ \mathcal{O}(N^{3\a-3})\bigg),
\end{split}
\end{equation}
which gives our desired Taylor expansion of the moment generating function conditioned on the event $\Sigma_N$.

\subsection{Boundedness of the moment generating function}
Let us briefly recall some notations:
\begin{align*}
\alpha_t = & p \nu_t= p \binom{N-t-1}{d-2}, \\
\Delta_{t+1} = & A_{t+1} - A_t - D_{t+1}, \\
A_t = & A_{t-1} + \eta_t -1, \\
D_{t+1} = & \E[\eta_{t+1}| \mathcal{F}_{t}] ,\\
\beta_t &= \prod_{i=1}^t (1-\a_i),  \\
S_t&=  \sum_{i=1}^t \frac{1}{\beta_i}\Delta_i.
\end{align*}
Note that $\beta_t/\beta_i= \prod_{j=i+1}^t (1-\a_j)$ is between 0 and 1.
Moreover recall that from Lemma \ref{lem_conVara.s} we obtain that 
$$
\Var(\Delta_t \vert \Fc_{t-1})\leq M
$$
for all $t$. 
Using the fact that $U_i \le N$ for all $1\leq i\leq \gamma_N$  analogously to \eqref{equ_mgfEST2}, we can derive the following expansion for all $0\leq i\leq \gamma_N$ and any fixed $t$
\begin{equation*} 
\begin{split}
& \E\left[\exp\left(t N^{\a-1} \frac{\beta_{\gamma(N)}} {\beta_{\gamma(i)+1}}\Delta_{\gamma(i)+1} \right) \bigg|\mathcal{F}_{\gamma(i)}\right] \\
=& 1+ \frac {t^2}{2}N^{2\a-2} \frac{\beta_{\gamma(N)}^2} {\beta_{\gamma(i)+1}^2} \Var\left(\Delta_{\gamma(i)+1} \vert \mathcal{F}_{\gamma(i)}\right) + \mathcal{O}(t^3 N^{3\a-3})\\
\leq& 1+\frac {t^2}{2}\frac{\beta_{\gamma(N)}^2} {\beta_{\gamma(i)+1}^2}  M N^{2\a-2}  + \mathcal{O}(t^3 N^{3\a-3})\\
\leq& \exp\left( \frac {t^2}{2} C N^{2\a-2} \right),
\end{split}
\end{equation*}
where $C>0$ is some constant large enough. Therefore, repeating the above expansion for all $1\leq i\leq \gamma(N)$ and employing the martingale difference sequence property as in \eqref{equ_mgfEST1} yield
\begin{equation}\label{eq_BoundMGF}
\begin{split}
&\E \left(\exp\left(t \frac{N^\a}{N} \beta_{\gamma(N)} S_{\gamma(N)}\right)\right)\\
\leq& \E\left( \exp\left(t \frac{N^\a}{N} \beta_{\gamma(N)} S_{\gamma(N)-1}\right)\E\left[\exp \left(t  \frac{N^\a}{N} \Delta_{\gamma(N)}\right) \bigg|\mathcal{F}_{\gamma(N)-1}\right]\right) \\
\leq& \exp\left( \frac {t^2}{2} \tilde{C}N^{2\a-1} \right),
\end{split}
\end{equation}
where $\tilde{C}>0$ is some large enough constant.

\subsection{The goal}
With above preparation, we are finally ready to prove: 
\begin{proof}[Proof of Theorem \ref{MDPmart}]
On one hand, by expanding the moment generating function in \eqref{eq_TaylorGammaN}, taking the logarithm and dividing by $N^{2\a -1}$, we see that 
\begin{equation}\label{equa_MgfSigmaN}
\begin{split}
&
\lim_{N \to \infty}\frac 1 {N^{2\a -1}}\log \E \left(\exp\left(t \frac{N^\a}{N}\beta_{\gamma(N)} S_{\gamma(N)}\right)\right)\\
\geq&
\lim_{N \to \infty}\frac 1 {N^{2\a -1}}\log \E \left[\exp\left(t \frac{N^\a}{N}\beta_{\gamma(N)} S_{\gamma(N)}\right) \mathds{1}_{\Sigma_N} \right]\\
=&
\lim_{N \to \infty}  \frac {t^2}2 \left[ \frac{1}{N} 
\sum_{i=0}^{\gamma(N)-1} \frac{\beta_{\gamma(N)}^2}{\beta_{i+1}^2} \left(\l(d-2)\left(1-\frac{i}{N}\right)^{d-3}\frac{u_i^2}{N^2}+ \l \left(1-\frac {i}{N}\right)^{d-2 }\frac {u_i}{N}\right) \right]\\
=& \frac {t^2}{2}\left(\lambda(1-\overline{\rho_\l})^2-\lambda_* (1-\overline{\rho_\l})+\overline{\rho_\l}(1-\overline{\rho_\l}) \right),
\end{split}
\end{equation}
where we used the abbreviation $\overline{\rho_\l}= \rho_{\l,d}$ and the last line follows by the asymptotics for 
$$
\sum_{i=0}^{\gamma(N)-1} \frac{\beta_{\gamma(N)}^2}{\beta_{i+1}^2} \left(\l(d-2)\left(1-\frac{i}{N}\right)^{d-3}\frac{u_i^2}{N^2}+ \l \left(1-\frac {i}{N}\right)^{d-2 }\frac {u_i}{N}\right)
$$
in \cite[(23)]{bollobas_riordan2}. Indeed, by $\gamma(N)= \overline{\rho_\l} N + \zeta N^\a$ and the explicit formula of $u_i$ in \eqref{eq_ui}, we have $u_i\leq cN$ for some $c>0$ uniformly in $1\leq i\leq N$. Moreover, $\beta_i$ is a constant between 0 and 1. Thus, each summand in
$$
\sum_{i=\overline{\rho_\l}N}^{\overline{\rho_\l}N+\zeta N^\a} \frac{\beta_{\gamma(N)}^2}{\beta_{i+1}^2} \left(\l(d-2)\left(1-\frac{i}{N}\right)^{d-3}\frac{u_i^2}{N^2}+ \l \left(1-\frac {i}{N}\right)^{d-2 }\frac {u_i}{N}\right)
$$
is uniformly bounded. Therefore, it follows
\begin{align*}
&\sum_{i=0}^{\gamma(N)-1} \frac{\beta_{\gamma(N)}^2}{\beta_{i+1}^2} \left(\l(d-2)\left(1-\frac{i}{N}\right)^{d-3}\frac{u_i^2}{N^2}+ \l \left(1-\frac {i}{N}\right)^{d-2 }\frac {u_i}{N}\right)\\
-& \sum_{i=0}^{\overline{\rho_\l}N-1} \frac{\beta_{\gamma(N)}^2}{\beta_{i+1}^2} \left(\l(d-2)\left(1-\frac{i}{N}\right)^{d-3}\frac{u_i^2}{N^2}+ \l \left(1-\frac {i}{N}\right)^{d-2 }\frac {u_i}{N}\right)=o(N).
\end{align*}
Thus
\begin{align*}
& \frac{1}{N}
\sum_{i=0}^{\overline{\rho_\l}N-1} \frac{\beta_{\gamma(N)}^2}{\beta_{i+1}^2} \left(\l(d-2)\left(1-\frac{i}{N}\right)^{d-3}\frac{u_i^2}{N^2}+ \l \left(1-\frac {i}{N}\right)^{d-2 }\frac {u_i}{N}\right)+o(1)
\end{align*}
yields the claim.

On the other hand, with the bound for moment generating function derived in \eqref{eq_BoundMGF} and the fact $\frac 12 < \a <\b <1$, it holds

\begin{equation}\label{equa_MgfSigmaNc}
\begin{split}
&\lim_{N \to \infty}\frac 1 {N^{2\a -1}}\log \E \left[\exp\left(t \frac{N^\a}{N}\beta_{\gamma(N)} S_{\gamma(N)}\right) \mathds{1}_{\Sigma_N^c} \right] \\
\leq& \lim_{N \to \infty}\frac 1 {N^{2\a -1}}\log \left( \exp\left( \frac {t^2}{2} \tilde{C} N^{2\a-1} \right) \P(\Sigma_N^c) \right)\\
\leq& \lim_{N \to \infty}\frac 1 {N^{2\a -1}}\left( \frac {t^2}{2} \tilde{C} N^{2\a-1} + \log\P(\Sigma_N^c) \right)\\
\leq& \lim_{N \to \infty}\frac 1 {N^{2\a -1}}\left( \frac {t^2}{2} \tilde{C} N^{2\a-1} - \eta N^{2\b-1} \right)\\
=& -\infty.
\end{split}
\end{equation}
Therefore, we can conclude 
\begin{align*}
\lim_{N \to \infty}\frac 1 {N^{2\a -1}}\log \E\,\bigg( \exp \bigg( t \,\frac{N^\a}N \beta_{\gamma(N)} & S_{\gamma(N)} \bigg)  \bigg)\\
\leq  \lim_{N \to \infty}\frac 1 {N^{2\a -1}} \log\max \bigg \lbrace \E \bigg[\exp \bigg(t \frac{N^\a}{N} & \beta_{\gamma(N)} S_{\gamma(N)}\bigg) \mathds{1}_{\Sigma_N} \bigg],\\
 \E \bigg[ \exp \bigg(t \frac{N^\a}{N} & \beta_{\gamma(N)} S_{\gamma(N)}\bigg) \mathds{1}_{\Sigma_N^c} \bigg] \bigg\rbrace \\
=\frac {t^2}{2}\big(\lambda(1-\overline{\rho_\l})^2-\lambda_* (1-\overline{\rho_\l}) + \overline{\rho_\l}(1-&\overline{\rho_\l}) \big).
\end{align*}
Here the last but one line follows by combining the equations \eqref{equa_MgfSigmaN} and \eqref{equa_MgfSigmaNc}. Since the parabola $ct^2/2$ satisfies all the assumptions on the moment generating function in the G\"artner-Ellis theorem \cite[Theorem 2.3.6]{dembozeitouni}, therefore, its Legendre transform $x^2/2c$ is then the rate function in Theorem \ref{MDPmart}.
\end{proof}

\section{Choice of $k_N$}
Following the idea we sketched in the previous sections, we seek a smart choice for $k_N$ such that the union of the connected components of the first $k_N$ vertices does not essentially differ from  the giant component with overwhelming probability. This union will be called $\Ckn$ in the sequel, its size will be denoted by $\Ckna$. 

Let us first recall two very useful results from the literature: It has been shown that $|\Cmax|$ concentrates on $\overline{\rho_\l}N$ and the second largest component, denoted by $\Cs$ is unlikely to be large as well.

\begin{remark}
The notation $f(N)=\Omega(g(N))$ is used for $\exists N_0\in \mathbb{N}$, $C>0$ so that $f(N)\geq Cg(N)$ for $N\geq N_0$.
\end{remark}

\begin{theorem}\cite[Theorem 4]{bollobas_riordan3}
\label{thm_largeDev}
With the same assumptions as in Theorem \ref{main_theorem}, i.e.\ set $\l=1+\epsilon$ with $\epsilon=\mathcal{O}(1)$ as well as $\epsilon^3N\to\infty$. If $\omega=\omega(N) \to\infty$ and $\omega=\Oc\left( \sqrt{\epsilon^3N}\right)$, then
$$
\P \left( \big| |\Cmax|-\overline{\rho_\l}N \big|\geq \omega\sqrt{N/\epsilon} \right) 
=\exp\left( -\Omega(\omega^2) \right).
$$
Furthermore, if $L=L(n)$ fulfills $\epsilon^2L\to\infty$ and $L=\Oc(\epsilon N)$. Then there exists $C>0$ such that the second largest component $\Cs$ in $\Gc(N,p)$ satisfies
\begin{equation}\label{equa_second}
\P\left(| \Cs| >L \right)\leq C\frac{\epsilon N}{L}\exp(-\epsilon^2L/C).
\end{equation}
for large enough $N$.
\end{theorem}

\underline{\textbf{Choice of $k_N$:}} we set
\begin{align}\label{choiceKn}
k_N= N^{\gamma} \qquad \text{ for a } 2\a-1<\gamma<\a,
  \end{align}
(recall that $\a\in(\frac{1}{2},1)$). 
The upper bound of $k_N$ will be discussed more precisely in Section \ref{sec_final}, whereas the lower bound arises from the fact that the union of connected components starting with $k_N$ vertices should contain the giant component with overwhelming probability:

\begin{proposition}\label{prop_Cmax>Ckn}
For any $\frac 12 < \a <1$,
$$
\P\left( |\Cmax|>\Ckna \right)\leq \exp(-Ck_N),
$$
for some $C>0$ large enough.
\end{proposition}
\begin{proof}
Note that
\begin{align*}
&\P\left( |\Cmax|>\Ckna \right)
= \P\left( \forall 1\leq i\leq k_N\colon i\notin\Cmax \right)\\
=& \P\left( 1\notin\Cmax \right)\P\left( 2\notin\Cmax\mid 1\notin\Cmax \right)\cdots \P\left( k_N\notin\Cmax\mid 1,\dots,k_{N-1}\notin\Cmax \right).
\end{align*}
For all $\xi>0$, it holds 
\begin{align*}
&\P\left( 1\notin\Cmax \right)\\
=& \P\left( 1\notin\Cmax,\, \left||\Cmax|-\overline{\rho_\l}N\right|>\xi N \right) + \P\left( 1\notin\Cmax,\, \left||\Cmax|-\overline{\rho_\l}N\right|\leq\xi N \right)\\
\leq& \P\left(  \left||\Cmax|-\overline{\rho_\l}N\right|>\xi N \right) + \P\left( 1\notin\Cmax,\, \left||\Cmax|-\overline{\rho_\l}N\right|\leq\xi N \right).
\end{align*}
Let us set $0<\xi<\overline{\rho_\l}$ and apply the large deviation bound given in Theorem \ref{thm_largeDev} with $\omega=\xi \sqrt{\vep^3 N}$, then the first term turns to be 

\begin{equation*}
\P\left(  \left||\Cmax|-\overline{\rho_\l}N\right|>\xi N \right)\leq \exp(-c\xi^2\epsilon^3 N)
\end{equation*}
for some constant $c>0$. This converges to $0$ by assumption on $\vep$.

On the other hand, note that on the event $\left\lbrace \overline{\rho_\l}N-\xi N \leq |\Cmax| \leq \overline{\rho_\l}N + \xi N \right\rbrace$, there are at most $N-(\overline{\rho_\l}N-\xi N)$ vertices which are not contained in $\Cmax$. Then we can bound the second term by
\begin{equation*}
\P\left( 1\notin\Cmax,\, \left||\Cmax|-\overline{\rho_\l}N\right|\leq\xi N \right)
\leq 1-\left(\overline{\rho_\l}-\xi \right).
\end{equation*}
Therefore, we obtain 
\begin{align*}
\P\left( 1\notin\Cmax \right) &\leq 1-\left(\overline{\rho_\l}-\xi \right)+\exp(-c\epsilon^3\xi^2 N),\\
\P\left( j\notin\Cmax\mid 1,\dots,j-1\notin\Cmax \right) &\leq \P\left( 1\notin\Cmax \right) ,
\end{align*}
for all $2\leq j\leq k_N$. Consequently, there exists an appropriate $C>0$ such that 
\begin{align*}
\P\left( |\Cmax|>\Ckna \right)
\leq& \Big( \P\left( 1\notin\Cmax \right) \Big)^{k_N}\\
\leq& \left( 1-\left(\overline{\rho_\l}-\xi \right)+\exp(-c\epsilon^3\xi^2N) \right)^{k_N}\\
\leq& \exp\left( -C(\overline{\rho_\l}-\xi) \right)^{k_N}\\
\leq& \exp(-Ck_N).
\end{align*}
\end{proof}

\section{A moderate deviations principle for $\Ckna$}
In this section we are going to show an MDP for $\Ckna$, in other words, the size of the union of the connected components, if we set the first $k_N$ vertices active in the exploration process and $k_N$ is of the right size. We will prove this MDP using the MDP for the martingale part of the exploration process derived in Theorem \ref{MDPmart}. In the next section we will see, that if $k_N$ is large enough, $\Ckna$ and $|\Cmax|$ only differ by an amount that is negligible on the moderate deviations scale.

\begin{theorem}\label{MDPCkn}
Consider a probability for the presence of a hyperedge as in \eqref{pdef} with $\l=1+\epsilon$ as in Theorem \ref{main_theorem}. Take $k_N= N^{\gamma}$ for $2\a-1<\gamma<\a$ as in \eqref{choiceKn}. Then for any $\frac 12 < \a <1$ and $y>0$ we have that
$$
\lim_{N \to \infty}\frac 1 {N^{2\a-1}}\log \P\left(\left|\Ckna - \overline{\rho_\l}N\right| >y N^{\a}\right)= -J(y)
$$
where $J$ is given by \eqref{rategraph}, i.e.\
$$
J(y)= I\left( y\left( 1-\l_*\right) \right)=\frac{y^2\left( 1-\l_*\right)^2}{2c}.
$$
and $I(\cdot)$ as well as $c$ is given explicitly in Theorem \ref{MDPmart}.
\end{theorem}
\begin{rem}
Due to the topological structure of $\R$ Theorem \ref{MDPCkn} is indeed a moderate deviations principle, see \cite[Section 3.7]{dembozeitouni}.
\end{rem} 

We will break up the proof of Theorem \ref{MDPCkn} into several lemmas.

\subsection{The upper bounds}\label{sec_upper}
\begin{lemma}\label{lem_uppermy+}
In the situation and with the notation of Theorem \ref{MDPCkn} we have for any $\frac 12 < \a <1$ and  $y>0$ that
$$
\lim_{N \to \infty}\frac 1 {N^{2\a-1}}\log \P\left(\Ckna >y N^{\a}+ \overline{\rho_\l}N\right)\le -J(y).
$$
\end{lemma}
\begin{proof}
Let $y N^{\a}+ \overline{\rho_\l}N=:m_y^+$. Firstly, recall the approximation process $\tilde A_t:= x_t+ \beta_t S_t$, and the trajectory of $x_t$ given by \eqref{eq_trajectoryxt}, and we see
\begin{equation}\label{eq_precisext}
 x_{m_y^+}
= f(m_y^+)+\Oc(1)
= Ng_{d,\l}(\frac{m_y^+}{N})+\Oc(1),
\end{equation}
where $g_{d,\l}$ is given by \eqref{eq_def.g}:
\begin{align*}
g_{d,\l}(x)= 1-x-\exp\left(-\frac \l{d-1}\left(1-(1-x)^{d-1}\right)\right).
\end{align*}
We define
\begin{equation*}
h( y)
=\exp\left(-\frac \l{d-1}\left(1-(1-yN^{\a-1}-\overline{\rho_\l})^{d-1}\right)\right).
\end{equation*}
Note that
\begin{align*}
h'(y)
= \exp\left(-\frac \l{d-1}\left(1-(1-yN^{\a-1}-\overline{\rho_\l})^{d-1}\right)\right) \left(-\l (1-yN^{\a-1}-\overline{\rho_\l})^{d-2}N^{\a-1} \right) ,
\end{align*}
and
\begin{align*}
h^{''}(y)
=& \exp\left(-\frac \l{d-1}\left(1-(1-yN^{\a-1}-\overline{\rho_\l})^{d-1}\right)\right) \big(\l^2 (1-yN^{\a-1}-\overline{\rho_\l})^{2(d-2)}N^{2\a-2}\\ 
&+\l(d-2) (1-yN^{\a-1}-\overline{\rho_\l})^{d-3}N^{2\a-2} \big) .
\end{align*}
Thus we may expand $h$ in $y=0$ to obtain
\begin{align*}
 h (y) 
=& \exp\left(-\frac \l{d-1}\left(1-(1-\overline{\rho_\l})^{d-1}\right)\right) \left(1-\l (1-\overline{\rho_\l})^{d-2} yN^{\a-1} \right) + \Oc(N^{2\a-2}) .
\end{align*}
Inserting the above calculation into $g_{d,\l}$ yields
\begin{align*}
 & g_{d,\l}(yN^{\a-1}+\overline{\rho_\l})\\
=& 1-yN^{\a-1}-\overline{\rho_\l}- h(y)\\
=& -yN^{\a-1}+\exp\left(-\frac \l{d-1}\left(1-(1-\overline{\rho_\l})^{d-1}\right)\right) \left( \l (1-\overline{\rho_\l})^{d-2}yN^{\a-1} \right)+ \Oc(N^{2\a-2}) \\
=& -yN^{\a-1}+ \l (1-\overline{\rho_\l})^{d-1}yN^{\a-1} + \Oc(N^{2\a-2}),
\end{align*}
where we used the fact that $g_{d,\l}(\overline{\rho_\l})=0$ by \eqref{rho_alt}. Therefore, we conclude by the definition of $x_t$ in \eqref{eq_precisext} 
\begin{equation}\label{eq_xmy+}
\begin{split}
x_{m_y^+}
=& N g_{d,\l}(yN^{\a-1}+\overline{\rho_\l}) + \Oc(1) \\
=& -y\left( 1-\l(1-\overline{\rho_\l})^{d-1} \right)N^\a + \Oc(N^{2\a-1})+ \Oc(1) \\
=& -y\left( 1-\l_*\right)N^\a + o(N^\a) + \Oc(1),
\end{split}
\end{equation}
where in the last line we used the abbreviation $\l_*= \l \left( 1- \rho_{d,\l} \right)^{d-1}$ in \eqref{def_lambda*}. Finally, recall that $\tilde{A}$ is the approximation process given by $\tilde{A}_t=x_t+\beta_tS_t$ in \eqref{Approx.Proc} and $\E\tilde{A}_t=x_t$ (as well as \eqref{eq_simMDP1}, \eqref{eq_simMDP2}). We observe 
\begin{equation*}
\begin{split}
\P\left(\Ckna >m_y^+ \right)
=& \P\left(\forall m\le m_y^+:\; A_m>0 \right)\\
\le& \P\left( A_{m_y^+}>0 \right)\\
=& \P\left( \frac{A_{m_y^+}-\E A_{m_y^+}}{N^\a}>-\frac{\E A_{m_y^+}}{N^\a} \right)\\
\sim& \P\left( \frac{\tilde{A}_{m_y^+}-\E \tilde{A}_{m_y^+}}{N^\a}>-\frac{\E\tilde{A}_{m_y^+}}{N^\a} \right)\\
=& \P\left(\frac{\beta_{m_y^+} S_{m_y^+}}{N^\a}> \frac{-x_{m_y^+} }{N^\a} \right)\\
=& \P\left(\frac{\beta_{m_y^+} S_{m_y^+}}{N^\a}> y\left( 1-\l_*\right) + o(1) \right).
\end{split}
\end{equation*}
Applying Theorem \ref{MDPmart}, it yields
\begin{align*}
&\lim_{N \to \infty}\frac 1 {N^{2\a-1}}\log \P\left(\Ckna >m_y^+ \right) \\
\leq& \lim_{N \to \infty}\frac 1 {N^{2\a-1}}\log \P\left(\frac{\beta_{m_y^+} S_{m_y^+}}{N^\a}> y\left( 1-\l_*\right) + o(1)   \right)\\
=& -I\left(y\left( 1-\l_*\right) \right) = -J(y).
\end{align*}
\end{proof}

\begin{lemma}\label{lemma5.4}
In the situation and with the notation of Theorem \ref{MDPCkn} we have for any $\frac 12 < \a <1$ and $y>0$ that
$$
\lim_{N \to \infty}\frac 1 {N^{2\a-1}}\log \P\left(\Ckna < -y N^{\a}+\overline{\rho_\l}N \right)\le -J(y).
$$
\end{lemma}
\begin{proof}
Analogously to $m_y^+$, let us define $m_y^-:=-y N^{\a}+\overline{\rho_\l}N$. We observe for each $\zeta$ with
$\frac{1}{2}<\zeta<1$ (which will be chosen more explicitly later) we obtain
\begin{align}
\P\left(\Ckna < m_y^- \right)
=& \P\left(\exists m<m_y^-: A_m=0\right) \nonumber \\
\leq& \sum_{m= -y N^{\zeta}+\overline{\rho_\l}N}^{m_y^-}\P\left( A_m=0\right) +  \P\left(\exists m< -y N^{\zeta}+\overline{\rho_\l}N \colon A_m=0\right) \nonumber\\
=& \sum_{m= -y N^{\zeta}+\overline{\rho_\l}N}^{m_y^-}\P\left( A_m=0\right) +  \P\left(\Ckna < -y N^{\zeta}+\overline{\rho_\l}N\right). \label{Upper2Est1}
\end{align}
In particular, the second term in \eqref{Upper2Est1} is negligible on the chosen moderation deviation scale. 
We pick $\omega=y\sqrt{\epsilon^3 N^{2\zeta-1}}$ and one can verify that $\omega=\Oc(\sqrt{\epsilon^3 N})$.

Indeed, it is implied by Theorem \ref{thm_largeDev} that for some constant $C>0$
\begin{align*}
& \lim_{N \to \infty}\frac 1 {N^{2\a-1}}\log\P(|\Cmax|<  -y N^{\zeta}+\overline{\rho_\l}N )\\
\leq & \lim_{N \to \infty}\frac 1 {N^{2\a-1}}\log\P(\left| |\Cmax|-\overline{\rho_\l}N \right|>  y N^{\zeta} )\\
\leq& \lim_{N \to \infty}\frac 1 {N^{2\a-1}}\log\P(\left| |\Cmax|-\overline{\rho_\l}N \right|> \omega\sqrt{N/\epsilon} )
\leq \lim_{N \to \infty}-\frac{Cy^2\epsilon^3N^{2\zeta-1} }{N^{2\a-1}}
=-\infty,
\end{align*}
Here, for the last equality to hold we need to choose $\zeta$ appropriately. This is done as follows: Recall that we require the condition \eqref{cond:tau}, i.e.
$ \epsilon^3 N^\tau\to\infty$ with $\tau=\min\lbrace \frac{1}{2}, 2-2\a-\iota \rbrace$ for a given $\iota>0$ small enough. Then we distinguish the following cases:  
\begin{enumerate}
\item If $2-2\a-\iota<\frac{1}{2}$, i.e.\ if $\a>3/4-\frac{\iota}{2}$, we get $\epsilon^3N^{2-2\a-\iota}\to\infty$. We set $\zeta> \max\lbrace 1-\frac{\iota}{2},\a\rbrace$, and thus obtain $2\zeta-2\a> 2-2\a-\iota$. This ensures
$$
\epsilon^3N^{2\zeta-2\a} \to\infty.
$$
\item If $2-2\a-\iota>\frac{1}{2}$ i.e.\ if $\a<3/4-\frac{\iota}{2}$, we see $\epsilon^3N^{1/2}\to\infty$. In this case define $\zeta> \a+1/4$, which implies $2\zeta-2\a> \frac{1}{2}$. Hence, it follows $\epsilon^3N^{2\zeta-2\a} \to\infty$.
\item If $2-2\a-\iota=\frac{1}{2}$, any of the above choices for $\zeta$ can be applied.
\end{enumerate}

Now we fix $y$ and $\zeta$ satisfying the above conditions.
For $m\in {[ -y N^{\zeta}+\overline{\rho_\l}N, m_y^- ]}$, we can find a $\delta$ with $\alpha\leq \delta\leq \zeta$ such that it holds $m=m_\delta:=-y N^{\delta}+\overline{\rho_\l}N$.
Note that $\delta$ will depend on $N$.

We distinguish again the following two cases.
Firstly, consider the set $M_1$ and $M_2$ defined below (to simplify notation, assume that sets $M_1$ and $M_2$ defined below are sets of integers, to avoid the irrelevant rounding):
$$
M_1:=\left\lbrace m\in {[ -y N^{\zeta}+\overline{\rho_\l}N, m_y^- ]} \colon  \frac{m-\overline{\rho_\l}N}{m_y^- -\overline{\rho_\l}N}\overset{N\to\infty}{\longrightarrow}\infty \right\rbrace.
$$
Note that $M_1$ contains exactly those $m$ for which 
$\liminf\delta=\liminf\delta_N>\alpha$, when we write $  m$ in the form $m_\delta$ as above.
Applying Theorem \ref{MDPmart} for $m_{\delta}= -y N^{\delta}+\overline{\rho_\l}N$ in the role of $\gamma(N)$ yields that $\beta_{m_{\delta}}S_{m_{\delta}}$ satisfies an MDP with speed $N^{2\delta-1}$ and rate function $I(x)=x^2/2c$, where $c$ was given explicitly in \eqref{eq_c}.

By \eqref{eq_simMDP1} and \eqref{eq_simMDP2}, we apply Theorem \ref{MDPmart} to the summands in the first term of \eqref{Upper2Est1} to obtain for all $m=m_\delta\in M_1$ that
\begin{align*}
&\lim_{N \to \infty}\frac 1 {N^{2\a-1}}\log \P\left( A_{m_\delta}=0\right)
\leq \lim_{N \to \infty}\frac 1 {N^{2\a-1}}\log \P\left( A_{m_\delta}\leq0\right)\\
\sim& \lim_{N \to \infty}\frac 1 {N^{2\a-1}}\log \P\left( \tilde{A}_{m_\delta}\leq 0\right)\\
\leq& \lim_{N \to \infty}\frac 1 {N^{2\a-1}}\log \P\left(\frac{ \beta_{m_\delta}S_{m_\delta}}{N^\delta}\leq  \frac{-x_{m_\delta} }{N^\delta} \right)\\
\leq& \lim_{N \to \infty}\frac {N^{2\delta-1}}{N^{2\a-1}} \left[ \frac{1}{{N^{2\delta-1}}}\log \P\left( \frac{\beta_{m_\delta}S_{m_\delta}}{N^\delta} \leq -\frac{x_{m_\delta}}{N^\delta} \right)\right]\\
=& -\infty.
\end{align*}
The last step follows, since $x_{m_\delta}/N^\delta \in \R$ such that the MDP for $\beta_{m_\delta}S_{m_\delta}$ holds. Indeed, we can expand $x_{m_\delta}$ analogously to \eqref{eq_xmy+} to see that $x_{m_\delta}$ is of the order $N^\delta$. 

Secondly, define 
$$
M_2:=\left\lbrace m\in {[ -y N^{\zeta}+\overline{\rho_\l}N, m_y^- ]} \colon  \frac{m-\overline{\rho_\l}N}{m_y^- -\overline{\rho_\l}N}\overset{N\to\infty}{\longrightarrow} -\mathrm{const.} \right\rbrace.
$$
Note that $M_2$ contains exactly those $m$ for which 
$\lim\delta=\lim\delta_N=\alpha$, when we write $m$ in the form $m_\delta$ as above. Applying Theorem \ref{MDPmart} together with \eqref{eq_simMDP1}, \eqref{eq_simMDP2}, it holds for all $m=m_\delta\in M_2$  that 
\begin{align*}
&\lim_{N \to \infty}\frac 1 {N^{2\a-1}}\log \P\left( A_{m_\delta}=0\right)
\leq \lim_{N \to \infty}\frac 1 {N^{2\a-1}}\log \P\left( A_{m_\delta}\leq0\right)\\
\sim& \lim_{N \to \infty}\frac 1 {N^{2\a-1}}\log \P\left( \tilde{A}_{m_\delta}\leq 0\right)\\
\leq& \lim_{N \to \infty}\frac {N^{2\delta-1}}{N^{2\a-1}} \left[ \frac{1}{{N^{2\delta-1}}}\log \P\left( \frac{\beta_{m_\delta}S_{m_\delta}}{N^\delta} \leq -\frac{x_{m_\delta}}{N^\delta} \right)\right]\\
=& -I\left(-y\left( 1-\l_*\right) \right) = -J(y).
\end{align*}

Therefore, we conclude
\begin{align*}
 & \lim_{N \to \infty}\frac 1 {N^{2\a-1}}\log \P\left(\Ckna < m_y^- \right)\\
\leq& \lim_{N \to \infty}\frac 1 {N^{2\a-1}}\log \max_{m\in {[ -y N^{\zeta}+\overline{\rho_\l}N,\, m_y^- ]}}   \P\left( A_m=0\right)  \\
\leq& \lim_{N \to \infty}\frac 1 {N^{2\a-1}} \max_{m_\delta\in {[ -y N^{\zeta}+\overline{\rho_\l}N,\, m_y^- ]}} \log\P\left( A_{m_\delta}=0\right) 
\leq -J(y).
\end{align*}
\end{proof}

\subsection{The lower bounds}\label{sec_lower}
In order to derive the corresponding lower bounds for our MDP, we need some preparations to get familiar with the properties of the exploration processes. 

\subsubsection{Recap of exploration process}
Let
$$ 0=t_0<t_1<t_2<\cdots <t_l=N $$
enumerate the event $\lbrace t\colon A_t-A_{t-1}=-1 \rbrace$, which are the moments where the exploration starts with a new component of the hypergraph. 
Let
$$
\mathcal{C}_t = \big| \lbrace i\colon 0\leq i<t,\, A_i-A_{i-1}=-1 \rbrace \big|
$$
be the number of components which have been explored by time $t$.

Furthermore, we define the random walk
$$
X_t=A_t - \Cc_t. 
$$
Recall the definition of $\l=1+\epsilon$ in Theorem \ref{main_theorem} that we have $\epsilon=\Oc(1)$ as well as $ \epsilon^3 N^\tau\to\infty$ with $\tau=\min\lbrace \frac{1}{2}, 2-2\a-\iota \rbrace$ for a given $\iota>0$ small enough. (This implies $\epsilon^3N\to\infty$.) Next fix the function $\omega=\omega(N)$ satisfying
\begin{equation}\label{equa_omega}
\omega=\omega(N)\to \infty \qquad \text{and} \qquad \omega\leq C\sqrt{\epsilon^3N}
\end{equation}
for some constant $C>0$. We define
$$
t^*_0:=\omega\sqrt{N/\epsilon}.
$$

Furthermore, let us denote the number of components completely explored within time $t^*_0$ by
$$
Z:=-\inf\lbrace X_t\colon t\leq t^*_0\rbrace.
$$
Finally, on the event $\lbrace \Cmax\subseteq\Ckn\rbrace$, set
\begin{align*}
T_0 =& \inf\lbrace t\colon X_t=- Z\rbrace, \quad \mbox{and }\\
T_1 =& \inf\lbrace t\colon X_t=- Z-1\rbrace,
\end{align*}
where $T_0$ is the time point at which the last component within time $t^*_0$ is completely explored, while $T_1$ is the time when we finish exploring the next component. Simply by definition, we have $T_0 \leq t^*_0 <T_1$. Now, ignoring the irrelevant rounding to integers let us set
$$
t^*_1 = \overline{\rho_\l}N.
$$

\subsubsection{The component $\Cc_{1,0}$}
Denote the component by $\Cc_{1,0}$, which we explore from time $T_0+1$ to time $T_1$. 
Recall in the supercritical regime, there is a unique giant component. 
Bollob\'as and Riordan show in \cite[Lemma 16]{bollobas_riordan3} that  with probability $1-\exp\left(-\Omega(w^2)\right)$ (where $\omega$ is given in \eqref{equa_omega})  on the event $\lbrace \Cmax\subseteq\Ckn \rbrace$ the component $\Cc_{1,0}$ is the unique giant component of $\Gc^d(N,p)$. Moreover, 
the formula for the size $\Cc_{1,0}$ (conditioned on $\lbrace \Cmax\subseteq\Ckn \rbrace$) is then given by
\begin{equation}\label{eq_sizeCc10}
|\Cc_{1,0}|= t^*_1 + \frac{\tilde{A}_{t^*_1}}{1-\l_*}
\end{equation}
in terms of the constructed approximation process $\tilde{A}_t$, see \cite[(21)]{bollobas_riordan1}.

\begin{lemma}
In the situation and with the notation of Theorem \ref{MDPCkn} we have for any $\frac 12 < \a <1$ and  $y>0$ that
$$
\lim_{N \to \infty}\frac 1 {N^{2\a-1}}\log \P\left(\Ckna >y N^{\a}+ \overline{\rho_\l}N\right)\ge -J(y).
$$
\end{lemma}
\begin{proof}
As implied by Proposition \ref{prop_Cmax>Ckn} we have
\begin{align*}
\P \left( \lbrace \Cmax\nsubseteq\Ckn \rbrace \right)
\leq \P\left( |\Cmax|>\Ckna \right) \leq \exp(-Ck_N).
\end{align*}
By our construction, on the event $\lbrace \Cmax\subseteq\Ckn \rbrace$, the component $\Cc_{1,0}$ exists and it satisfies
\begin{align*}
& \lim_{N \to \infty}\frac 1 {N^{2\a-1}}\log \P\left(\Ckna > y N^{\a}+ \overline{\rho_\l}N \right) \\
\geq& \lim_{N \to \infty}\frac 1 {N^{2\a-1}}\log \P\left(\Ckna > y N^{\a}+ \overline{\rho_\l}N, \,\Cmax\subseteq\Ckn \right) \\
\geq& \lim_{N \to \infty}\frac 1 {N^{2\a-1}}\log \P\left( |\Cc_{1,0}|  > y N^{\a}+ \overline{\rho_\l}N, \,\Cmax\subseteq\Ckn \right).
\end{align*}
Indeed, by our construction, the component $\Cc_{1,0}$ exists and is contained in the union of connected components $\Ckn$ on the event $\lbrace \Cmax\subseteq\Ckn \rbrace$.

Therefore, by inserting the approximation for $|\Cc_{1,0}|$ in \eqref{eq_sizeCc10} we observe
\begin{align*}
&  \P\left(\Ckna > y N^{\a}+ \overline{\rho_\l}N \right) \\
\geq &  \P\left( |\Cc_{1,0}|  > y N^{\a}+ \overline{\rho_\l}N\mid \Cmax\subseteq\Ckn \right) \P\left( \Cmax\subseteq\Ckn  \right) \\
\geq& \P\big( t^*_1 + \frac{\tilde{A}_{t^*_1}}{1-\l_*}  > y N^{\a}+ \overline{\rho_\l}N \mid \Cmax\subseteq\Ckn \big)\big( 1-\exp\left(-Ck_N\right)\big).
\end{align*}
Finally, using the notation $t^*_1 = \overline{\rho_\l}N$, we arrive at 
\begin{align*}
& \lim_{N \to \infty}\frac 1 {N^{2\a-1}}\log \P\left(\Ckna > y N^{\a}+ \overline{\rho_\l}N \right) \\
\geq& \lim_{N \to \infty}\frac 1 {N^{2\a-1}} \bigg\lbrack \log \P\left(t^*_1+\frac{\tilde{A}_{t^*_1}}{1-\l_*} > y N^{\a}+ \overline{\rho_\l}N\right) +\log \left(1-\exp(-Ck_N)\right) \bigg\rbrack\\
\geq& \lim_{N \to \infty}\frac 1 {N^{2\a-1}} \left[ \log \P\left( \frac{\beta_{t^*_1}S_{t^*_1}}{ N^{\a}}> y(1-\l_*)  \right) \right]  \\
\geq& -I\left((y\left( 1-\l_*\right) ) \right) =-J(y).
\end{align*}
\end{proof}

\begin{lemma}
In the situation and with the notation of Theorem \ref{MDPCkn} we have for any $\frac 12 < \a <1$ and  $y>0$ that
$$
\lim_{N \to \infty}\frac 1 {N^{2\a-1}}\log \P\left(\Ckna <-y N^{\a}+ \overline{\rho_\l}N\right)\ge -J(y).
$$
\end{lemma}
\begin{proof}
Again let $m_y^-=-y N^{\a}+\overline{\rho_\l}N$ and by \eqref{eq_xmy-} we obtain
\begin{equation*}
x_{m_y^-}
= y\left( 1-\l_*\right)N^\a + o(N^\a) + \Oc(1).
\end{equation*}
Therefore, using Theorem \ref{MDPmart} we obtain
\begin{align*}
&\lim_{N \to \infty}\frac 1 {N^{2\a-1}}\log \P\left(\Ckna < m_y^- \right)
\leq \lim_{N \to \infty}\frac 1 {N^{2\a-1}}\log \P\left(\Ckna \leq m_y^- \right) \\
=& \lim_{N \to \infty}\frac 1 {N^{2\a-1}}\log \P\left(\exists m\leq m_y^-: A_m=0 \right) \\
\geq& \lim_{N \to \infty}\frac 1 {N^{2\a-1}}\log  \P\left(A_{m_y^-} \leq 0 \right) \\
\sim& \lim_{N \to \infty}\frac 1 {N^{2\a-1}}\log \P\left( \frac{\beta_{m_y^-} S_{m_y^-}}{N^\a} \leq  -\frac{x_{m_y^-}}{N^\a}   \right)\\
=& \lim_{N \to \infty}\frac 1 {N^{2\a-1}}\log \P\left( \frac{\beta_{m_\delta}S_{m_\delta}}{N^\a } \leq -y \left(1-\l*\right)+o(1) \right)\\
=& -I\left(-y\left( 1-\l_*\right) \right) = -J(y).
\end{align*}
\end{proof}

\section{Proof of Theorem \ref{main_theorem}\label{sec_final}: a moderate deviations principle for $|\Cmax|$}
\subsection{Compare $|\Cmax|$ and $\Ckna$}
In this subsection we improve Proposition \ref{prop_Cmax>Ckn}. We show, that if we allow an error term $r_N=o(N^\a)$ we obtain a bound on the probability of the event $ |\Cmax|+r_N< \Ckna$ that is negligible on the moderate devation scale. Hence, the upper bound of $k_N$ given by \eqref{choiceKn} can be obtained.

We remind the reader of $k_N=N^\gamma$ for $2\a-1<\gamma<\a$ , where $\frac 12 < \a <1$ is chosen.
\begin{lemma}
\label{lem_errorESTrn}
Let $r_N=N^\xi$ for $\gamma<\xi<\a$. Then
\begin{equation*}
 \P\left( |\Cmax|+r_N< \Ckna\right)\leq \exp\left(-Mk_N \right) + \exp\left( o( N^{2\a-1}) \right)
\end{equation*}
for some constant $M>0$ small enough.
\end{lemma}
\begin{proof}
Let us denote the $i$-th largest component of $\Gc^d(N,p)$ by $L_i$, whose size is given by $l_i=|L_i|$. Then, because $\Ckna$ can at most be as large as the union of the $k_N$ largest components, we get for each $\delta>0$ that
\begin{equation*}
\begin{split}
&\P\left( |\Cmax|+r_N< \Ckna\right)\\
\leq& \P\left( |\Cmax|+r_N< |\Cmax|+\big| \bigcup_{i=2}^{k_N} L_i\big| \right)\\
\leq& \P\left( r_N<\big| \bigcup_{j=2}^{k_N} L_i\big|,\, |\Cmax|> \delta N \right) + \P\left( |\Cmax|\leq \delta N \right)\\
\leq &\sum_{\substack{a_2>\dots>a_{k_N}\\ a_2+\dots+a_{k_N}> r_N}}\P( l_2=a_2,\,\dots,\, l_{k_N}=a_{k_N},\,|\Cmax|>\delta N ) + \P\left( |\Cmax|\leq \delta N \right).
\end{split}
\end{equation*}

 
Define $\delta N:=\overline{\rho_\l}N-\epsilon N$ with $\vep$ again as defined in Theorem \ref{main_theorem}. We obtain by Theorem \ref{thm_largeDev} with $\omega= \sqrt{\epsilon^3 N}$ that
\begin{equation}\label{eq_CmaxDelta}
\begin{split}
\P\left( |\Cmax|\leq \delta N \right)
\leq \P\left( ||\Cmax|-\overline{\rho_\l}N|\geq  \omega \sqrt{N/\epsilon} \right)
\leq \exp\left(-c \epsilon^3 N \right),
\end{split}
\end{equation}
where $c>0$ is some constant.
Recall the conditions $\epsilon=\mathcal{O}(1)$ as well as $ \epsilon^3 N^\tau\to\infty$ with $\tau=\min\lbrace \frac{1}{2}, 2-2\a-\iota \rbrace$, for a fixed small $\iota>0$ given in Theorem \ref{main_theorem}. We obtain 
\begin{equation}\label{eq_CmaxdeltaN}
\begin{split}
\frac{1}{N^{2\a-1}}\log\P\left( |\Cmax|\leq \delta N \right)
\leq& \frac{1}{N^{2\a-1}} \log \exp\left(- c\epsilon^3 N \right) \\
\leq& -\frac{c\epsilon^3 N }{N^{2\a-1}}
\overset{N\to\infty}{\longrightarrow} -\infty.
\end{split}
\end{equation}

Moreover, note that for fixed $a_2>\dots>a_{k_N}$ with $a_2+\dots+a_{k_N}> r_N$,  it holds
\begin{equation*}
\begin{split}
&\P( l_2=a_2,\,\dots,\, l_{k_N}=a_{k_N},\, |\Cmax|>\delta N )\\
=& \prod_{j=2}^{k_N} \P\left( l_j= a_j\mid l_2=a_2,\,\dots,\, l_{j-1}=a_{j-1},\, |\Cmax|>\delta N \right)\,\P(|\Cmax|>\delta N)\\
\leq& \prod_{j=3}^{k_N} \P\left( l_j= a_j\mid l_2=a_2,\,\dots,\, l_{j-1}=a_{j-1},\, |\Cmax|>\delta N \right)\, \P\left( \lbrace l_2\geq a_2\rbrace  \cap\lbrace |\Cmax|>\delta N\rbrace \right)\\
\leq& \prod_{j=3}^{k_N} \P\left( l_j\geq a_j\mid l_2=a_2,\,\dots,\, l_{j-1}=a_{j-1},\, |\Cmax|>\delta N \right)\, \P(l_2\geq a_2).
\end{split}
\end{equation*}

On one hand, we obtain from \eqref{equa_second} that for some constant $c,C>0$ 
$$
\P(l_2>a_2)\leq C\frac{\epsilon N}{a_2}\exp\left( -c\epsilon^2 a_2\right),
$$
where the $\epsilon$ was given by the branching factor $\l=1+\epsilon$ of the original hypergraph  $\Gc^d(N,p)$.

On the other hand, for each $j\in\lbrace 3,\dots,k_N\rbrace$, the hypergraph $\Gc^d(N,p)$ after removing the components $L_2, \dots ,L_{j-1}$ and $\Cmax$, conditioned on $\lbrace |\Cmax|>\delta N\rbrace$ by \cite[Lemma 8.1]{bollobas_riordan3} (or common sense)  is again a hypergraph $\Gc^d(N-s_j, p)$, where $s_j\leq (1-\delta)N$. Recall that $$p=\frac{\l (d-2)!}{N^{d-1}}=\frac{\l_j(r-2)!}{(N-s_j)^{d-1}}$$ in $\Gc^d(N-s_j, p)$ is characterised by its branching factor 
$$
\l_j=\left(1-\frac{s_j}{N}\right)^{d-1}\l,
$$
where we have $\l=1+\epsilon$ by assumptions. 
Now let us denote by $\vep_j$ the following quantity: 
$$
\epsilon_j =1-\left(1-\frac{s_j}{N} \right)^{d-1}(1+\epsilon).
$$
Then we can arrive at 
\begin{equation}\label{eq_lambdaj}
\l_j=1-\epsilon_j.
\end{equation}

Since $s_j\leq(1-\delta)N$, we see by \cite[(8.3)]{bollobas_riordan3} that
\begin{align*}
&\epsilon_j \leq 1-\delta^{d-1}(1+\epsilon), \\
&c_\delta \epsilon\leq \epsilon_j\leq C_\delta \epsilon
\end{align*}
for some constanst $c_\delta, C_\delta>0$ that depend on $\delta$ but not on $j$. 

Therefore, the branching factor $\l_j$ given in \eqref{eq_lambdaj} of the new hypergraph $\Gc^d(N-s_j,p)$, belongs to the subcritical regime. From \cite[Theorem 2]{bollobas_riordan3} we obtain a large deviation bound for the largest component in $\Gc^d(N-s_j,p)$, that is to say, the $L_j$ of $\Gc^d(N,p)$. We thus obtain
\begin{equation*}
\begin{split}
&\P\left( l_j \geq a_j\mid l_2=a_2,\,\dots,\, l_{j-1}=a_{j-1},\, |\Cmax|>\delta N \right)\\
\leq& C\frac{\epsilon_j N}{a_j}\exp\left( -c\epsilon_j^2 a_j\right)\\
\leq& C_\delta\frac{\epsilon N}{a_j}\exp\left( -c_\delta\epsilon^2 a_j\right).
\end{split}
\end{equation*}

Hence, we get fianally for $M>0$ small enough
\begin{equation}\label{eq_kthlargest}
\begin{split}
& \sum_{\substack{a_2>\dots>a_{k_N}\\ a_2+\dots+a_{k_N}> r_N}}\P( l_2=a_2,\,\dots,\, l_{k_N}=a_{k_N},\,|\Cmax|>\delta N )\\ 
\leq&\,  r_N^{k_N}\frac{ (C_\delta\epsilon N)^{ k_N}}{\prod_{j=2}^{k_N}a_j}\exp\left( -c_\delta\epsilon^2 \sum_{j=2}^{k_N}a_j\right)\\
\leq & 
\exp\left( -c_\delta\epsilon^2 r_N + k_N\log (C_\delta \epsilon r_N N)\right)\\
\leq& \exp\left(-Mk_N \right).
\end{split}
\end{equation}
The last step follows, since not only the fact $k_N=o(r_N)$ is implied by $r_N=N^\xi$ with $\gamma<\xi<\a$ and $k_N=N^\gamma$ with $2\a-1< \gamma<\a$ in \eqref{choiceKn}; but also $k_N=o(N)$ holds.
Hence, \eqref{eq_kthlargest} together with \eqref{eq_CmaxDelta}, \eqref{eq_CmaxdeltaN} yields 
\begin{equation*}
 \P\left( |\Cmax|+r_N< \Ckna\right)\leq \exp\left(-Mk_N \right) + { \exp\left( o( N^{2\a-1}) \right)}
\end{equation*}
by adjusting the constant $M>0$.
\end{proof}

\subsection{Proof of Theorem \ref{main_theorem}}
As we described before, by an appropriate choice of $k_N$, $\Ckna$ and $|\Cmax|$ only differ by an amount that is negligible on the moderate deviations scale.
\begin{proof}
Note that we pick $r_N=N^\xi$ for $\gamma<\xi<\a$, in particular, it satisfies $r_N=o\left(N^\a \right)$. Now for all $y>0$, we estimate the upper tail by applying Lemma \ref{lem_errorESTrn} with $M$ given there to obtain 
\begin{align*}
 & \P\left( ||\Cmax|-\overline{\rho_\l}| >y N^{\a}\right)\\
=& \P\left( ||\Cmax|-\overline{\rho_\l}| >y N^{\a},\, ||\Cmax|-\Ckna|\leq r_N \right)\\
& +\P\left( ||\Cmax|-\overline{\rho_\l}| >y N^{\a},\, ||\Cmax|-\Ckna|> r_N \right) \\
\leq& \P\left( |\Ckna-\overline{\rho_\l}| >y N^{\a} + o\left(N^\a \right)\right) +  \exp\left(-Mk_N \right) +\exp\left( o( N^{2\a-1}) \right).
\end{align*}
It suffices to apply the MDP for $\Ckna$ derived in Theorem \ref{MDPCkn}, it yields
\begin{align*}
& \lim_{N \to \infty}\frac 1 {N^{2\a-1}}\log \P\left( ||\Cmax|-\overline{\rho_\l}| > y N^{\a}\right)\\
\leq& \lim_{N \to \infty}\frac 1 {N^{2\a-1}}\log \Big\lbrack \P\left( |\Ckna-\overline{\rho_\l}| > y N^{\a} + o\left(N^\a \right)\right) + \exp\left(-Mk_N \right) \Big\rbrack\\
\leq& \lim_{N \to \infty}\frac 1 {N^{2\a-1}}\log\max \Big\lbrace \P\left( |\Ckna-\overline{\rho_\l}| > y N^{\a}+ o\left(N^\a \right)\right),\, \exp\left(-Mk_N \right) \Big\rbrace\\
\leq& -J(y).
\end{align*}
Similarly, for the lower tail, it satisfies for all $y>0$ that
\begin{equation*}
\begin{split}
& \P\left( |\Ckna - \overline{\rho_\l}| >y N^{\a}\right)\\
=& \P\left( |\Ckna-\overline{\rho_\l}| >y N^{\a},\, ||\Cmax|-\Ckna|\leq r_N \right)\\
& +\P\left( |\Ckna - \overline{\rho_\l}| >y N^{\a},\, ||\Cmax|-\Ckna|> r_N \right) \\
\leq& \P\left( ||\Cmax|-\overline{\rho_\l}| >y N^{\a} + o\left(N^\a \right)\right) +  \exp\left(-Mk_N \right) +\exp\left( o( N^{2\a-1}) \right).
\end{split}
\end{equation*}
Again by Theorem \ref{MDPCkn}, we arrive at
\begin{equation*}
\begin{split}
-J(y)=& \lim_{N \to \infty}\frac 1 {N^{2\a-1}}\log \P\left( |\Ckna - \overline{\rho_\l}| >y N^{\a}\right)\\
\leq& \lim_{N \to \infty}\frac 1 {N^{2\a-1}}\log\max \Big\lbrace \P\left( ||\Cmax|-\overline{\rho_\l}| >y N^{\a} + o\left(N^\a \right)\right) ,\, \exp\left(-Mk_N \right) \Big\rbrace \\
\leq& \lim_{N \to \infty}\frac 1 {N^{2\a-1}}\log \P\left( ||\Cmax|-\overline{\rho_\l}| >y N^{\a}\right) .
\end{split}
\end{equation*}
Altogether, the claim
$$
\lim_{N \to \infty}\frac 1 {N^{2\a-1}}\log \P\left( ||\Cmax|-\overline{\rho_\l}| >y N^{\a} \right) =-J(y)
$$
follows for all $y>0$.
\end{proof}

\section*{Acknowledgements}
Research of the authors was funded by the Deutsche Forschungsgemeinschaft (DFG, German Research Foundation) under Germany's Excellence Strategy EXC 2044 -390685587, Mathematics M\"unster: Dynamics -Geometry -Structure.

\bibliographystyle{alpha}
\bibliography{LiteraturDatenbank}
\end{document}